\documentclass[11pt,a4paper,twoside]{article}
\usepackage{amsfonts} 
\usepackage{amssymb}
\usepackage{amsmath, amsthm, array, ifthen}
\usepackage{mathrsfs} 
\usepackage{caption} 
\usepackage{latexsym}
\usepackage{comment}
\usepackage{mathtools,nccmath}
\usepackage{euscript}

\DeclarePairedDelimiterX\Ioo[2]{\rbrack}{\lbrack}{#1,#2}
\DeclarePairedDelimiterX\Iof[2]{\rbrack}{\rbrack}{#1,#2}
\DeclarePairedDelimiterX\Ifo[2]{\lbrack}{\lbrack}{#1,#2}
\DeclarePairedDelimiterX\Iff[2]{\lbrack}{\rbrack}{#1,#2}
\DeclarePairedDelimiterX\Jfo[2]{\lbrack}{\lbrack}{#1;#2}
\DeclarePairedDelimiterX\Jff[2]{\lbrack}{\rbrack}{#1;#2}

\usepackage{enumerate}

\usepackage[np]{numprint} 
\npdecimalsign{\ensuremath{.}}

\usepackage{enumitem}

\usepackage{xfrac}
\usepackage{graphicx}

\usepackage{lmodern}

\usepackage[utf8]{inputenc} 

\usepackage{color}

\usepackage[colorlinks]{hyperref}

\newcommand{\R}{\mathbb{R}}
\newcommand{\C}{\mathbb{C}}
\newcommand{\Z}{\mathbb{Z}}

\newcommand{\D}{\mathcal{D}}

\newcommand{\tri}{\begin{scriptsize}$\triangleright$\end{scriptsize}}

\newcommand{\mell}[2][f]{\mathcal{M}(f)} 

\newcommand{\diff}{\mathop{}\mathopen{}\mathrm{d}}

\newcommand{\sdfrac}[2]{\mbox{\small$\displaystyle\frac{#1}{#2}$}}

\theoremstyle{theorem}
\newtheorem{prop}{Proposition}
\newtheorem{prop/not}{Proposition/notation}
\newtheorem{thm}{Theorem}[section]

\newtheorem{cor}{Corollary}[section]
\newtheorem{lem}{Lemma}

\theoremstyle{remark}
\newtheorem{rem}{Remark}

\newcommand{\ioe}{\leqslant}
\newcommand{\soe}{\geqslant}

\usepackage[usenames,dvipsnames,svgnames,table]{xcolor}

\providecommand{\keywordsubject}[1]{\textbf{2020 Mathematics Subject Classification :} \: #1}
\providecommand{\keywords}[1]{\textbf{Keywords :} #1}

\title{On a Smoothed Walfisz Divisor Problem}

\usepackage[left=1.5cm,right=1.5cm,top=1.5cm,bottom=1.5cm]{geometry}
\author{ Olivier Bordellès and Florian Daval}
\date{}
\begin{document}

\maketitle

\footnote{\keywordsubject{11A25, 11N37, 11L07}}

\footnote{\keywords{Walfisz divisor problem, Chowla-Walum sums}}

\begin{abstract}
This work is in the spirit of our previous investigation on a smooth Dirichlet divisor problem, where we now replace the Dirichlet divisor function $\tau$ by the sum-of-divisors function $\sigma$. We prove a totally explicit asymptotic formula for the sum of $\sigma(n)$ twisted by the weight $1-x/n$, which enables us to eliminate the difficult part in the classical average order of $\sigma(n)$. As a corollary, we deduce the convergence of an integral dealing with the error term in the Walfisz divisor problem. We also provide an appendix containing the necessary explicit results derived from the mean value theorem and the Euler-Maclaurin summation formula.
\end{abstract}

\section{Introduction}

\subsection{Main results}

\noindent
In \cite{bor26}, we studied the sum
$$S_{\tau} (x) := \sum_{n \leqslant x} \tau(n) \left( 1 - \mfrac{x}{n} \right) $$
whose weight $1-x/n$ allowed us to eliminate the very difficult sum $\sum_{n \ioe \sqrt{x}} \psi \left( \frac{x}{n} \right)$, where $\psi(t) := \{ t \} - \frac{1}{2}$  is the first Bernoulli periodic function and $\{ t \}$ is the fractional part of $t$. The main objective of this work is a similar study of the sum 
$$S_{\sigma} (x) := \sum_{n \leqslant x} \sigma(n) \left( 1 - \mfrac{x}{n} \right) $$
where $\sigma(n) := \displaystyle \medop \sum_{d \mid n} d$ is the sum-of-divisors function. Classical calculations, such as in \cite[pp. 472-473]{bor20} for instance, show that
$$\sum_{n \leqslant x} \sigma(n) = \sdfrac{x^2 \zeta(2)}{2} - x \sum_{n \ioe \sqrt{x}} \mfrac{1}{n} \psi \left( \mfrac{x}{n}\right) - \mfrac{x^2}{2} \sum_{n > x} \mfrac{1}{n^2} + \mfrac{1}{2} \sum_{n\leqslant x} \left( B_2 \left( \left\lbrace \mfrac{x}{n} \right\rbrace \right) - \mfrac{1}{6} \right) $$
where $B_2 \left( \left\lbrace t \right\rbrace \right)  := \left\lbrace t \right\rbrace^{2} - \left\lbrace t \right\rbrace + \frac{1}{6}$ is the second Bernoulli periodic function, and the trivial bounds applied to the above sums yield
$$\sum_{n \leqslant x} \sigma(n) = \sdfrac{x^2 \zeta(2)}{2} + O^{\star} \left( \tfrac{1}{2} x \log x + \tfrac{9}{8} x + \tfrac{1}{2} \right )$$
for all $x \geqslant 1$. In \cite{wal63}, Walfisz was the first to get a non-trivial upper bound in this problem. To estimate sums of the shape $\medop \sum_{n \sim N} \mfrac{1}{n} \psi \left( \mfrac{x}{n}\right)$, he used a van der Corput type estimate for large $N$, and a Vinogradov type result for smaller ones. This allowed him to derive the following improvement
$$\sum_{n \leqslant x} \sigma(n) = \sdfrac{x^2 \zeta(2)}{2} + O \Bigl( x (\log x)^{2/3} \Bigr)$$
which is still the best result to date for this problem. Therefore, we called it the \textit{Walfisz divisor problem}. As can be seen in Lemma~\ref{le:main}, when the weight $1-x/n$ is included, the sum $\medop \sum_{n \ioe \sqrt{x}} \mfrac{1}{n} \psi \left( \mfrac{x}{n}\right)$ disappears and is replaced by the following easier one
$$\sum_{n \leqslant \sqrt{x}} \left(1+\mfrac{n^2}{x} \right) B_2 \left( \left\lbrace \mfrac{x}{n} \right\rbrace \right).$$

\begin{thm}
\label{th:sigma} 
For all real numbers $x \geqslant 1$, define
\begin{align*}
\D(x)&:=\sum_{n \ioe x}\sigma(n)- \mfrac{x^2 \zeta(2)}{2} \\
d(x)&:=\sum_{n \ioe x} \mfrac{\sigma(n)}{n}-\left( x \zeta(2) - \tfrac{1}{2} \log x \right)\;.
\end{align*}
Then, for all $x \geqslant 650$, we have
$$\D(x)-xd(x) = \mfrac{x}{2} \left( \log 2 \pi + \gamma - 1 \right) + O^{\star} \Bigl( 16 \, x^{3/10} \log x \Bigr).$$
\end{thm}

\begin{rem}
Note that $ \tfrac{1}{2} \left( \log 2 \pi  + \gamma -1 \right) \doteq \np{0.7075} \dotsc $
\end{rem}

\noindent
As a counterpart of \cite[Corollary~1]{bor26}, we derive the next consequence from Theorem~\ref{th:sigma}. A related integral has been studied in \cite[Theorem~1]{zhai13}.

\begin{cor}
\label{cor:sigma}
For all $x \geqslant 650$,
$$\int_1^{x} \sdfrac{\D(t)}{t^{2}} \diff{t} < 0.$$
\end{cor}

\noindent
More surprisingly, the following corollary shows that the integral $\displaystyle \medint\int_1^{\infty} \mfrac{\D(t) + t/2}{t^{2}} \diff{t}$ converges even though it is known that $\D(x) = \Omega \left( x \log \log x \right) $. See for instance \cite{gron13}, \cite{wig14}, \cite{pet87} or \cite{ishi93} for more precise oscillation results.

\begin{cor}
\label{cor:sigma_2}
The integral $\displaystyle \medint \int_1^{\infty} \sdfrac{\D(t) + t/2}{t^{2}} \diff{t}$ converges and
$$\int_1^{\infty} \sdfrac{\D(t) + t/2}{t^{2}} \diff{t} = \mfrac{1}{2} \left(1 + \zeta(2) - \gamma - \log 2 \pi \right) \doteq \np{0.11492} \dotsc$$
\end{cor}

\subsection{Notations}

\noindent
We will use the first two Bernoulli polynomials $B_1(x) = x - \frac{1}{2}$ and $B_2(x)=x^2-x+\frac{1}{6}$, and their attached periodic Bernoulli functions $B_j(\{x\})$. It is customary to denote $B_1 (\{x\}) := \psi(x)$ the $1$st periodic Bernoulli function, and note that $\psi(x)^2 = B_2 (\{x\}) + \frac{1}{12}$. We also define $M_{\sigma}(x) := \frac{1}{2} \zeta(2) x^{2}$ and $ m_{\sigma}(x) := x \zeta(2) - \frac{1}{2} \log x$. With these notations, note that $\D(x)-xd(x) = S_{\sigma}(x) - M_{\sigma}(x) + x m_{\sigma}(x)$. Finally, $H(x):=\sum_{n \ioe x} \frac{1}{n}$ will be the harmonic sum and, for all $x \geqslant 1$, $\alpha \in \R_{> 1}$, $\beta \in \R_{\geqslant 0}$ and $j \in \Z_{\geqslant 1}$, the Chowla-Walum sums are defined as
$$G_{\alpha,\beta,j} (x) := \sum_{n \leqslant x^{1/\alpha}} n^{\beta} B_j \left( \left\lbrace \mfrac{x}{n} \right\rbrace \right).$$

\section{Proofs}

\subsection{Consequences of the Euler-Mclaurin summation formula}

\noindent
The following three lemmas are immediate consequences of Theorem~\ref{th:em_synthèse} with $k=2$ for \eqref{eq:harm_1} and $k=3$ for the others.

\begin{lem}[Harmonic numbers]
For all real numbers $z \geqslant 1$
\begin{equation}
   H(z) = \log z + \gamma - \mfrac{\psi(z)}{z} + O^{\star} \left( \mfrac{1}{8 z^2} \right) \label{eq:harm_1}
\end{equation}
and
\begin{equation}
   H(z) = \log z + \gamma - \mfrac{\psi(z)}{z} - \mfrac{B_2(\{z\})}{2z^{2}} + O^{\star} \left( \mfrac{1}{\pi^3 z^3} \right). \label{eq:harm_2}
\end{equation}
\end{lem}

\begin{lem}
\label{le:1/n^2}
For all real numbers $z \geqslant 1$
$$\medop \sum_{n > z} \mfrac{1}{n^2} = \mfrac{1}{z} + \mfrac{\psi(z)}{z^2} + \mfrac{B_2(\{z\})}{z^3} + O^\star \left( \mfrac{3}{\pi^{3}z^4} \right).$$
\end{lem}

\begin{lem}[Real Stirling]
\label{le:Stirling}
For all real numbers $z \geqslant 1$
$$\medop \sum_{n \leqslant z} \log n = z \log z - z + \tfrac{1}{2} \log 2 \pi -  \psi (z) \log z + \mfrac{B_2 \left( \{ z \} \right)}{2z}  + O^\star \left( \mfrac{1}{2\pi^{3}z^{2}} \right).$$
\end{lem}

\subsection{Explicit estimates for the Chowla-Walum sums}

\subsubsection{An explicit third derivative test}

\noindent
Recall that, for all $x \geqslant 1$, $\alpha \in \R_{> 1}$, $\beta \in \R_{\geqslant 0}$ and $j \in \Z_{\geqslant 1}$, we set
$$G_{\alpha,\beta,j} (x) := \sum_{n \leqslant x^{1/\alpha}} n^{\beta} B_j \left( \left\lbrace \mfrac{x}{n} \right\rbrace \right).$$

\noindent
We first state an explicit version of the third derivative test for exponential sums.

\begin{lem}
\label{le:3rd_test}
Let $A \in \Z_{\geqslant 0}$, $B \in \Z_{\geqslant 1}$, and $f \in C^3 \left[ A+1,A+B\right]$ such that there exist real numbers $0 < \lambda_3 < 1$ and $c_3 \geqslant 1$ such that, for all $x \in \left[ A+1,A+B\right]$, we have
$$\lambda_3 \leqslant \left| f^{\, \prime \prime \prime} (x) \right| \leqslant c_3 \lambda_3.$$
Then
$$\left| \sum_{A < n \leqslant A+B} e \left( f(n) \right) \right| < \left(8^{1/2} c_3^{1/2} + 2^{5/4} c_3^{1/4} \right)B \lambda_3^{1/6} + \left( 2 + 22^{1/2} c_3^{1/8} + 4 c_3^{-1/8} + 2^{3/4} c_3^{-1/4} \right) B^{1/2} \lambda_3^{-1/6}.$$
\end{lem}

\begin{proof}
One may suppose that $B\lambda_3^{1/3} \geqslant 1$ for, otherwise, $B^{1/2} \lambda_3^{-1/6} > B$. We apply \cite[Lemma~2.12]{pat21} with $\left( W_3,\lambda_3,\eta_3,N,L\right) \leftrightarrow \left( \lambda_3^{-1},c_3,\varepsilon_3,A,B\right) $, yielding
\begin{align*}
   \left| \sum_{A < n \leqslant A+B} e \left( f(n) \right) \right|^2 & \leqslant \left( B \lambda_3^{1/3} + \varepsilon_3 \right) \left( \alpha_3 B + \beta_3 \lambda_3^{-2/3} \right) \\
   & \leqslant B \lambda_3^{1/3} \left( 1 + \varepsilon_3 \right) \left( \alpha_3 B + \beta_3 \lambda_3^{-2/3} \right)
\end{align*}
since $B\lambda_3^{1/3} \geqslant 1$. We then majorize $\alpha_3$ and $\beta_3$ with the help of \cite[Remark~2.13]{pat21} and choose $\varepsilon_3 = (2c_3)^{-1/2}$, which allows us to obtain $\alpha_3 < 8 c_3$ and $\beta_3 < 22 c_3^{1/4}+4$. Reporting above, we get
$$\left| \sum_{A < n \leqslant A+B} e \left( f(n) \right) \right|^2 \leqslant \left( 8c_3 + 4 \sqrt{2c_3} \right) B^2 \lambda_3^{1/3} + \left(4 + 22 c_3^{1/4} + 11 \sqrt{2} c_3^{-1/4} + 2 \sqrt{2} c_3^{-1/2}\right) B \lambda_3^{-1/3} $$
giving the asserted result.
\end{proof}

\subsubsection{Application to the Chowla-Walum sums}

\noindent
We are now in a position to prove the main estimate of this section, which can be regarded as an explicit version of \cite[Theorem~2]{kane85}.

\begin{prop}
\label{pro:C_W}
Let $\alpha > 1$, $\beta \in \R_{\geqslant 0}$ and $j \in \Z_{\geqslant 2}$. Set $\Gamma_j := \mfrac{2 \eta(j)j!}{(2\pi)^{j}}$ where $\eta(j) = \zeta(j)$ if $j$ is even, $\eta(j) = 1$ otherwise. Then, for all $x \geqslant 1$, we have
$$\left| G_{\alpha,\beta,j} (x) \right| \leqslant \mfrac{2^{\beta + 5}}{\sqrt{\pi}}\Gamma_j  \Biggl( \Bigl( \zeta \left( j - \tfrac{1}{2}  \right) + \delta \Bigr) x^{\theta(\alpha,\beta)} + \zeta \left( j + \tfrac{1}{2} \right) x^{\frac{\beta}{\alpha} + \frac{3}{2\alpha} - \frac{1}{2}} \Biggr) \left(  \mfrac{c\log x}{\log 2} + 1 \right) $$
where
\begin{enumerate}
   \item[\tri] $\theta(\alpha,\beta) := \begin{cases} \frac{\beta}{\alpha} - \frac{1}{2\alpha} + \frac{1}{2} , & \mathrm{if} \ \beta \geqslant \frac{1}{2} \ \mathrm{and} \ 1 < \alpha < 3 \ \, ; \\ & \\ \frac{2 \beta}{5} + \frac{3}{10} , & \mathrm{if} \ 0 \leqslant \beta \leqslant \frac{1}{2} \ \mathrm{and} \ 1 < \alpha \leqslant \frac{5}{2} \, ; \end{cases}$
   \item[\tri] $ \left( \delta,c \right)  :=  \begin{cases} \left( \frac{\sqrt{\pi}}{2^{\beta+5}}, \mfrac{3-\alpha}{2\alpha(\beta+1)} \right) , & \mathrm{if} \ \beta \geqslant \frac{1}{2} \, ; \\ & \\ \Biggl( 2 \left( \zeta \left( j - \frac{1}{6}\right) + \zeta \left( j + \frac{1}{6}\right) \right) , \mfrac{1}{\alpha} \Biggr)  , & \mathrm{if} \ 0 \leqslant \beta \leqslant \frac{1}{2}. \end{cases}$
\end{enumerate}
\end{prop}

\begin{proof}
The case $\beta \geqslant \frac{1}{2}$ is \cite[Proposition~1]{bor26}. Now suppose $\beta \in \left[ 0, \frac{1}{2} \right]$ and $ 1 < \alpha \leqslant \frac{5}{2} $. Let $T \in \left[ 1,x^{1/\alpha}\right] $ be a parameter at our disposal, and write

$$G_{\alpha,\beta,j} (x) = \left( \sum_{n \leqslant T} + \sum_{T < n \leqslant x^{1/\alpha}} \right) n^{\beta} B_j \left( \left\lbrace \mfrac{x}{n} \right\rbrace \right)$$
so that
$$L(x)^{-1} \left| G_{\alpha,\beta,j} (x) \right|  \leqslant \max_{N \leqslant T} \left | G_{N,\beta,j}(x) \right | + \max_{T < N \leqslant x^{1/\alpha}} \left | G_{N,\beta,j}(x) \right |$$
where $L(x) := \frac{\log x}{\alpha \log 2} + 1$ and
$$G_{N,\beta,j}(x) := \sum_{N < n \leqslant 2N} n^{\beta} B_j \left( \left\lbrace \mfrac{x}{n} \right\rbrace \right).$$
As it was observed in \cite{bor26}, we have
$$\left | G_{N,\beta,j}(x) \right | \leqslant 2^{\beta+1} \Gamma_{j} N^{\beta} \, \sum_{m \geqslant 1} m^{-j} \max_{N \leqslant N_1 \leqslant 2N} \left| \sum_{N < n \leqslant N_1} e \left( \mfrac{mx}{n}\right) \right|$$
so that
\begin{align}
   L(x)^{-1} \left| G_{\alpha,\beta,j} (x) \right|  & \leqslant 2^{\beta+1} \Gamma_{j} \left\lbrace  \max_{N \leqslant T} \left( N^{\beta} \, \sum_{m \geqslant 1} m^{-j} \max_{N \leqslant N_1 \leqslant 2N} \left| \sum_{N < n \leqslant N_1} e \left( \mfrac{mx}{n}\right) \right| \right) \right. \notag \\
& \hspace*{1cm} \left. + \max_{T < N \leqslant x^{1/\alpha}} \left( N^{\beta} \, \sum_{m \geqslant 1} m^{-j} \max_{N \leqslant N_1 \leqslant 2N} \left| \sum_{N < n \leqslant N_1} e \left( \mfrac{mx}{n}\right) \right| \right) \right\rbrace \label{eq:C_W_1}
\end{align}
We now use Lemma~\ref{le:3rd_test} with $A=N$, $B=N_1-N \leqslant N$, $\lambda_3 = \frac{3mx}{8N^{4}}$ and $c_3 = 16$ for the first sum, and \cite[Lemma~8]{bor26} with $\lambda_2 = \frac{mx}{4N^3}$ and $c_2 = 8$ for the second one, which yields
\begin{align}
   & \max_{N \leqslant T} \left( N^{\beta} \, \sum_{m \geqslant 1} m^{-j} \max_{N \leqslant N_1 \leqslant 2N} \left| \sum_{N < n \leqslant N_1} e \left( \mfrac{mx}{n}\right) \right| \right) \notag \\
   & \hspace*{1cm} < \max_{N \leqslant T} \sum_{m \geqslant 1} m^{-j} \left( \np{13.65} (mx)^{1/6} N^{\beta + 1/3} + \np{14.5} (mx)^{-1/6} N^{\beta + 7/6} \right) \notag \\
   & \hspace*{2cm} \leqslant \np{13.65} \zeta \left( j - \mfrac{1}{6}\right)  x^{1/6} T^{\beta + 1/3} + \np{14.5} \zeta \left( j + \mfrac{1}{6}\right) x^{-1/6} T^{\beta + 7/6} \label{eq:C_W_2}
\end{align}
and
\begin{align}
   & \max_{T < N \leqslant x^{1/\alpha}} \left( N^{\beta} \, \sum_{m \geqslant 1} m^{-j} \max_{N \leqslant N_1 \leqslant 2N} \left| \sum_{N < n \leqslant N_1} e \left( \mfrac{mx}{n}\right) \right| \right) \notag \\
   & \hspace*{1cm} < \mfrac{16}{\sqrt{\pi}} \, \max_{T < N \leqslant x^{1/\alpha}} \sum_{m \geqslant 1} m^{-j} \left( (mx)^{1/2} N^{\beta - 1/2} +  (mx)^{-1/2} N^{\beta + 3/2} \right) \notag \\
   & \hspace*{2cm} \leqslant \mfrac{16}{\sqrt{\pi}} \left(  \zeta \left( j - \mfrac{1}{2}\right)  x^{1/2} T^{\beta - 1/2} +  \zeta \left( j + \mfrac{1}{2}\right) x^{\frac{\beta}{\alpha} + \frac{3}{2 \alpha} - \frac{1}{2}} \right).  \label{eq:C_W_3}
\end{align}
Reporting \eqref{eq:C_W_2} and \eqref{eq:C_W_3} in \eqref{eq:C_W_1} and choosing $T = x^{2/5}$ yield
$$L(x)^{-1} \left| G_{\alpha,\beta,j} (x) \right|  \leqslant 2^{\beta+1} \Gamma_{j} \left( C_j x^{\frac{2 \beta}{5} + \frac{3}{10}} + \mfrac{16}{\sqrt{\pi}} \zeta \left( j + \mfrac{1}{2}\right) x^{\frac{\beta}{\alpha} + \frac{3}{2 \alpha} - \frac{1}{2}} \right)$$
with
\begin{align*}
   C_j &:= \mfrac{16}{\sqrt{\pi}} \left( \mfrac{\np{13.65} \sqrt{ \pi}}{16} \zeta \left( j - \mfrac{1}{6}\right) + \mfrac{\np{14.5} \sqrt{ \pi}}{16} \zeta \left( j + \mfrac{1}{6}\right) + \zeta \left( j - \mfrac{1}{2}\right) \right) \\
   & < \mfrac{16}{\sqrt{\pi}} \left( 2 \zeta \left( j - \mfrac{1}{6}\right) + 2 \zeta \left( j + \mfrac{1}{6}\right) + \zeta \left( j - \mfrac{1}{2}\right) \right)
\end{align*}
as required.
\end{proof}

\begin{rem}
The case used in this paper is $\alpha = j = 2$ and $\beta \in \{0,2\}$. Since $\Gamma_2 = \frac{1}{6}$, we get for all $x \geqslant 1$
\begin{align}
   & \left| G_{2,0,2} (x) \right| < \left( \np{17.94} x^{3/10} + \np{4.04} x^{1/4} \right) \left( \mfrac{\log x}{\log 4} + 1 \right),  \notag \\
   & \left| G_{2,2,2} (x) \right| < \np{47.76} x^{5/4} \left( \mfrac{\log x}{\log 256} + 1 \right). \label{eq:borne_G}
\end{align}
\end{rem}

\subsubsection{Chowla-Walum sums and exponent pairs}

\noindent
When $0 \leqslant \beta \leqslant \frac{1}{2}$, the exponent $\frac{1}{10} \left( 4 \beta + 3 \right) $ was discovered by Kanemitsu and Sitaramachandrarao in \cite{kane85}, and improves the previous bound $ \frac{1}{3} (\beta + 1) $ obtained by the sole application of \cite[Lemma~8]{bor26}. To our knowledge, this is currently the best result that can be achieved in the explicit estimates setting. But this exponent is far from the conjectured value $\frac{\beta}{2} + \frac{1}{4}$ surmised by Chowla and Walum in \cite{chow65}. Nowadays, the improvements in this problem are direct consequences of improvements in exponent pairs. Such results have already been proved in the literature in the case $\alpha = 2$, see for instance \cite[Theorem~2]{zhai13}. We will prove here an alternate bound for the Chowla-Walum sums.

\begin{lem}
\label{le:C-W_exp_pairs}
Let $\alpha \in \R_{> 1}$, $\beta \in \R_{\geqslant 0}$ and $j \in \Z_{\geqslant 2}$. Let $(k,\ell) \ne (0,1)$ be an exponent pair. Then, for $x \geqslant e$ sufficiently large
$$(\log x)^{-1} G_{\alpha,\beta,j} (x) \ll_{\alpha,\beta,j} \ \begin{cases} x^{\frac{\beta}{\alpha}+\frac{k(\alpha-2)+\ell}{\alpha}}, & \mathrm{if} \ \beta + \ell - 2k \geqslant 0 \, \\ & \\ x^{\frac{k(\beta+1)}{2k+1-\ell}} + x^{k+\frac{\beta+\ell-2k}{\alpha}} + x^{\frac{\beta+2}{\alpha}-1}, & \mathrm{if} \ \beta + \ell - 2k \leqslant 0. \end{cases}$$
\end{lem}

\begin{rem}
It is not difficult to see that if $\alpha \geqslant 3$ and $k + \ell \geqslant 1$, then the trivial bound $x^{\frac{\beta+1}{\alpha}}$ is better than the two majorizations above.
\end{rem}

\begin{proof}
The case $ \beta + \ell - 2k \geqslant 0 $ is \cite[Theorem~2.2]{bor22}, so that we only treat the case $ \beta + \ell - 2k \leqslant 0 $. Note that, since $k \leqslant \frac{1}{2} \leqslant \ell$, we have $\beta \leqslant 2k - \ell \leqslant \frac{1}{2}$. As in the proof of Proposition~\ref{pro:C_W}, partial summation and the use of exponent pairs yield
\begin{align*}
   G_{\alpha,\beta,j} (x) & \ll_{\alpha,\beta,j} \ T^{\beta + 1} + \max_{T < N \leqslant x^{1/\alpha}} \Bigl( x^{k} N^{\beta + \ell - 2k} + x^{-1} N^{2+\beta} \Bigr) \log x \\
   & \ll _{\alpha,\beta,j} \ T^{\beta + 1} + \Bigl( x^{k} T^{\beta + \ell - 2k} + x^{\frac{\beta+2}{\alpha}-1} \Bigr) \log x
\end{align*}
with $1 \leqslant T \leqslant x^{1/\alpha}$, and where we used the condition $ \beta + \ell - 2k \leqslant 0 $. We then complete the proof using the Srinivasan optimisation lemma. See \cite[Lemma~5.6 p. 363]{bor20}, for instance.
\end{proof}

\noindent
We now focus on the case $\beta = 0$. By Lemma~\ref{le:C-W_exp_pairs}, we can see that it is in our best interest to choose an exponent pair $(k,\ell)$ satisfying $2k = \ell$. Now, note that if $\bigl(k,k+\frac{1}{2} \bigr)$ is an exponent pair, then the exponent pair $BA \bigl(k,k+\frac{1}{2} \bigr) = \left( \frac{2k+1}{4(k+1)}, \frac{4k+2}{4(k+1)}\right) $ satisfies this condition. We then infer  the following consequence.

\begin{lem}
Let $\alpha \in \R_{> 1}$ and $j \in \Z_{\geqslant 2}$. If $\bigl(k,k+\frac{1}{2} \bigr)$ is an exponent pair, then, for all $x \geqslant e$ sufficiently large, we have
$$G_{\alpha,0,j}(x) \ll \left( x^{\frac{2k+1}{4(k+1)}} + x^{\frac{2}{\alpha}-1} \right) \log x.$$
\end{lem}

\noindent
In particular, if the \textit{exponent pair conjecture} is true, i.e. if $(k,\ell) = \bigl( 0 , \frac{1}{2} \bigr) + \varepsilon$ is an exponent pair, then $G_{2,0,j}(x) \ll x^{1/4 + \varepsilon}$. Historically, the first exponent pair of the form $\bigl(k,k+\frac{1}{2} \bigr)$ to be discovered was the exponent pair $\left (\frac{1}{6},\frac{2}{3} \right) $, coming from Lemma~\ref{le:3rd_test}. Subsequently, the value of $k$ has been improved over the years by Huxley, Watt, Kolesnik and Bourgain, proving successively that $k = \frac{9}{56} + \varepsilon$, $k = \frac{89}{560} + \varepsilon$, $k = \frac{17}{108} + \varepsilon$, $k = \frac{89}{570} + \varepsilon$, $k = \frac{32}{205} + \varepsilon$, and $k = \frac{13}{84} + \varepsilon$ are admissible values for $k$. See \cite{trud24} for more details. In \cite{pet88}, the author chose $k = \frac{9}{56} + \varepsilon$ ; in \cite{zhai13}, the authors picked up $k = \frac{32}{205} + \varepsilon$. Taking $k = \frac{13}{84} + \varepsilon$ yields the best value to date.

\begin{cor}
Let $\alpha \in \R_{> 1}$ and $j \in \Z_{\geqslant 2}$. Then, for all $x \geqslant e$ sufficiently large and $\varepsilon > 0$, we have
$$G_{\alpha,0,j}(x) \ll x^{55/194 + \varepsilon} + x^{\frac{2}{\alpha}-1}  \log x.$$
\end{cor}

\noindent
As another example, using \cite[Theorem~1, (3), (i), p. 519]{ishi93} we get the following improvement of \cite[(5.7)]{pet88}.

\begin{cor}
\label{cor:main}
Let $T \in \R_{\geqslant 1}$. Then, for all $\varepsilon > 0$, we have
$$\int_{1}^{T} d(x) \diff{x} = - \mfrac{1}{2} \left (\gamma + \log 2 \pi \right) T  + O \left(  T^{55/194 + \varepsilon} \right).$$
\end{cor}

\subsection{Preparatory lemmas}

\begin{lem}
\label{le:prep_1}
For all $x \geqslant 1$
\begin{align}
   \sum_{n \leqslant \sqrt{x}} \ \sum_{m \leqslant x/n} (m+n) &= \mfrac{x^{2}}{2} \sum_{n \leqslant \sqrt{x}} \mfrac{1}{n^2} + x \left \lfloor \sqrt{x} \right \rfloor - \tfrac{1}{12} \left \lfloor \sqrt{x} \right \rfloor \left( 3\left \lfloor \sqrt{x} \right \rfloor + 4 \right) \notag  \\
   & \hspace*{1cm}  - \sum_{n \leqslant \sqrt{x}} \left( n+\mfrac{x}{n}\right) \psi \left( \mfrac{x}{n}\right) + \mfrac{1}{2} \, G_{2,0,2} (x) \label{eq:sum_1}
\end{align}
and
\begin{align}
   x \sum_{n \leqslant \sqrt{x}} \ \sum_{m \leqslant x/n} \mfrac{m+n}{mn} &= x^{2} \sum_{n \leqslant \sqrt{x}} \mfrac{1}{n^2} + x  \left \lfloor \sqrt{x} \right \rfloor \left( \log x + \gamma \right) - x \sum_{n \leqslant \sqrt{x}} \log n - \sum_{n \leqslant \sqrt{x}} \left( n + \mfrac{x}{n} \right) \psi \left( \mfrac{x}{n}\right) \notag \\
   & \hspace*{1cm}  - \mfrac{x}{2} H \left( \sqrt{x} \right) - \mfrac{1}{2x} \, G_{2,2,2} (x) + O^{\star} \left( \mfrac{1}{\pi^3} \right). \label{eq:sum_2}
\end{align}
\end{lem}

\begin{proof}
Straightforwardly, we have
\begin{align*}
    \sum_{n \leqslant \sqrt{x}} \ \sum_{m \leqslant x/n} (m+n) &= \sum_{n \leqslant \sqrt{x}} \left( \mfrac{1}{2} \left \lfloor \mfrac{x}{n} \right \rfloor \left \lfloor \mfrac{x}{n} + 1 \right \rfloor + n \left \lfloor \mfrac{x}{n} \right \rfloor \right) \\   
    &= \sum_{n \leqslant \sqrt{x}}  \left( \tfrac{1}{2} \left( \mfrac{x}{n} \right)^{2} - \mfrac{x}{n} \psi \left( \mfrac{x}{n} \right) + \tfrac{1}{2} B_2 \left( \left\lbrace \mfrac{x}{n} \right\rbrace \right) - \tfrac{1}{12} + n \left( \mfrac{x}{n} - \psi \left( \mfrac{x}{n}\right) - \tfrac{1}{2} \right)  \right) \\ 
    &= \mfrac{x^{2}}{2} \sum_{n \leqslant \sqrt{x}} \mfrac{1}{n^2} + x \left \lfloor \sqrt{x} \right \rfloor - \tfrac{1}{4} \left \lfloor \sqrt{x} \right \rfloor \left \lfloor \sqrt{x} + 1 \right \rfloor  - \tfrac{1}{12} \left \lfloor \sqrt{x} \right \rfloor \\
    & \hspace*{1cm} - \sum_{n \leqslant \sqrt{x}} \left( n+\mfrac{x}{n}\right) \psi \left( \mfrac{x}{n}\right) + \mfrac{1}{2} \sum_{n \leqslant \sqrt{x}} B_2 \left( \left\lbrace \mfrac{x}{n} \right\rbrace \right). 
\end{align*}
Similarly
\begin{align*}
   x \sum_{n \leqslant \sqrt{x}} \ \sum_{m \leqslant x/n} \mfrac{m+n}{mn} &= x \sum_{n \leqslant \sqrt{x}} \ \sum_{m \leqslant x/n} \left( \mfrac{1}{m} + \mfrac{1}{n} \right) \\
   &= x^{2} \sum_{n \leqslant \sqrt{x}} \mfrac{1}{n^2} - x \sum_{n \leqslant \sqrt{x}} \mfrac{1}{n} \psi \left( \mfrac{x}{n}\right) - \mfrac{x}{2} H \left( \sqrt{x} \right) + x \sum_{n \leqslant \sqrt{x}} H \left( \mfrac{x}{n}\right) 
\end{align*}
and since, by \eqref{eq:harm_2},
\begin{align*}
   x \sum_{n \leqslant \sqrt{x}} H \left( \mfrac{x}{n}\right) &= x \sum_{n \leqslant \sqrt{x}} \left( \log \mfrac{x}{n} + \gamma - \mfrac{n}{x} \psi \left( \mfrac{x}{n}\right) - \mfrac{n^{2}}{2x^{2}} B_2 \left( \left\lbrace \mfrac{x}{n} \right\rbrace \right) + O^{\star} \left( \mfrac{n^{3}}{\pi^{3} x^{3}}\right)  \right) \\
   &= x  \left \lfloor \sqrt{x} \right \rfloor \left( \log x + \gamma \right) - x \sum_{n \leqslant \sqrt{x}} \log n - \sum_{n \leqslant \sqrt{x}} n \psi \left( \mfrac{x}{n}\right) - \mfrac{1}{2x} \sum_{n \leqslant \sqrt{x}} n^{2} B_2 \left( \left\lbrace \mfrac{x}{n} \right\rbrace \right) + O^{\star} \left( \mfrac{1}{\pi^3} \right)
\end{align*}
we derive
\begin{align*}
   x \sum_{n \leqslant \sqrt{x}} \ \sum_{m \leqslant x/n} \mfrac{m+n}{mn} &= x^{2} \sum_{n \leqslant \sqrt{x}} \mfrac{1}{n^2} + x  \left \lfloor \sqrt{x} \right \rfloor \left( \log x + \gamma \right) - x \sum_{n \leqslant \sqrt{x}} \log n - \sum_{n \leqslant \sqrt{x}} \left( n + \mfrac{x}{n} \right) \psi \left( \mfrac{x}{n}\right) \\
   & \hspace*{1cm}  - \mfrac{x}{2} H \left( \sqrt{x} \right) - \mfrac{1}{2x} \, G_{2,2,2} (x) + O^{\star} \left( \mfrac{1}{\pi^3} \right). 
\end{align*}
\end{proof}

\begin{lem}
\label{le:prep_2}
For all $x \geqslant 1$,
\begin{align}
   & - \mfrac{x^{2}}{2} \sum_{n \leqslant \sqrt{x}} \mfrac{1}{n^2} -  x \left \lfloor \sqrt{x} \right \rfloor \Bigl( \log x + \gamma - 1 - H \left( \sqrt{x}\right) \Bigr) + x \sum_{n \leqslant \sqrt{x}} \log n \notag \\
   & \hspace*{1cm} = - \mfrac{x^2}{2} \zeta(2) + \mfrac{x^{3/2}}{2} + \mfrac{x}{4} \log x + \mfrac{x}{2} \left( \log (2 \pi) - 1 \right) + \mfrac{\sqrt{x}}{12} - \mfrac{3x}{2} \psi(\sqrt{x})  \notag \\
   & \hspace{2cm}  + \mfrac{\sqrt{x}}{2} \left( \psi(\sqrt{x}) + 3B_2 \left( \left\lbrace \sqrt{x} \right\rbrace \right) \right) + O^{\star}  \left( \mfrac{4}{\pi^{3}} + \mfrac{1}{12} \right). \label{eq:Stirling+harm+1/n^2}
\end{align}
\end{lem}

\begin{proof}
By Lemma~\ref{le:1/n^2}, we have
$$\mfrac{x^{2}}{2} \sum_{n \leqslant \sqrt{x}} \mfrac{1}{n^2} = \mfrac{x^2}{2} \zeta(2) - \mfrac{x^{3/2}}{2} - \mfrac{x \psi(\sqrt{x})}{2} - \mfrac{\sqrt{x}}{2} B_2 \left( \left\lbrace \sqrt{x} \right\rbrace \right) + O^{\star} \left( \mfrac{3}{2\pi^3} \right)$$
and using Lemma~\ref{le:Stirling}, we derive
$$x \sum_{n \leqslant \sqrt{x}} \log n = x^{3/2} \log \sqrt{x} - x^{3/2} + \tfrac{1}{2} x \log 2 \pi - x \psi \left( \sqrt{x}\right) \log \sqrt{x} + \tfrac{1}{2} \sqrt{x} B_2 \left( \{ \sqrt{x} \} \right)  + O^\star \left( \mfrac{1}{2\pi^{3}} \right)$$
so that
\begin{align}
    - \mfrac{x^{2}}{2} \sum_{n \leqslant \sqrt{x}} \mfrac{1}{n^2} + x \sum_{n \leqslant \sqrt{x}} \log n  &= - \mfrac{x^2}{2} \zeta(2) + \mfrac{x}{2} \log 2 \pi  \notag \\
    & \hspace*{0.25cm} + \left( \mfrac{x^{3/2}}{2} - \mfrac{x \psi \left( \sqrt{x}\right)}{2} \right) \left( \log x - 1 \right) + \sqrt{x} B_2 \left( \{ \sqrt{x} \} \right) + O^\star \left( \mfrac{2}{\pi^{3}} \right). \label{eq:Stirling}
\end{align}
Now since, by \eqref{eq:harm_2}
$$\log x + \gamma - 1 - H \left( \sqrt{x}\right) = \log \sqrt{x} - 1 + \mfrac{\psi(\sqrt{x})}{\sqrt{x}} + \mfrac{B_2 \left( \{ \sqrt{x} \} \right)}{2x} +  O^\star \left( \mfrac{1}{\pi^{3} x^{3/2}} \right)$$
we get, using the equality $\psi(\sqrt{x})^2 = B_2 \left( \{ \sqrt{x} \} \right) + \frac{1}{12}$
\begin{align}
   x \left \lfloor \sqrt{x} \right \rfloor \left( \log x + \gamma - 1 - H \left( \sqrt{x}\right) \right) &= x^{3/2} \log \sqrt{x} - x^{3/2} - \mfrac{x}{4}  \log x + \mfrac{x}{2} - \mfrac{\sqrt{x}}{12} - x \psi(\sqrt{x}) \log \sqrt{x} + 2x \psi(\sqrt{x}) \notag \\
   & \hspace*{1cm} - \mfrac{\sqrt{x}}{2} \left( \psi(\sqrt{x}) + B_2 \left( \left\lbrace \sqrt{x} \right\rbrace \right) \right) - \tfrac{1}{4} B_2 \left( \left\lbrace \sqrt{x} \right\rbrace \right) \left( 1 + 2 \psi(\sqrt{x})\right) + O^{\star}  \left( \mfrac{2}{\pi^{3}} \right) \notag \\
   & = x^{3/2} \log \sqrt{x} - x^{3/2} - \mfrac{x}{4}  \log x + \mfrac{x}{2} - \mfrac{\sqrt{x}}{12} - x \psi(\sqrt{x}) \log \sqrt{x} + 2x \psi(\sqrt{x}) \notag \\
   & \hspace*{1cm} - \mfrac{\sqrt{x}}{2} \left( \psi(\sqrt{x}) + B_2 \left( \left\lbrace \sqrt{x} \right\rbrace \right) \right) + O^{\star}  \left( \mfrac{2}{\pi^{3}} + \mfrac{1}{12} \right). \label{eq:Stirling-harm-2}
\end{align}
Substracting \eqref{eq:Stirling-harm-2} to \eqref{eq:Stirling}, we obtain
\begin{align*}
   & - \mfrac{x^{2}}{2} \sum_{n \leqslant \sqrt{x}} \mfrac{1}{n^2} -  x \left \lfloor \sqrt{x} \right \rfloor \Bigl( \log x + \gamma - 1 - H \left( \sqrt{x}\right) \Bigr) + x \sum_{n \leqslant \sqrt{x}} \log n \\
   & \hspace*{1cm} = - \mfrac{x^2}{2} \zeta(2) + \mfrac{x^{3/2}}{2} + \mfrac{x}{4} \log x + \mfrac{x}{2} \left( \log (2 \pi) - 1 \right) + \mfrac{\sqrt{x}}{12} - \mfrac{3x}{2} \psi(\sqrt{x})  \\
   & \hspace{2cm}  + \mfrac{\sqrt{x}}{2} \left( \psi(\sqrt{x}) + 3B_2 \left( \left\lbrace \sqrt{x} \right\rbrace \right) \right) + O^{\star}  \left( \mfrac{4}{\pi^{3}} + \mfrac{1}{12} \right) 
\end{align*}
as asserted.
\end{proof}

\subsection{Main Lemma}

\begin{lem}
\label{le:main}
For all $x \geqslant 1$,
$$\D(x)-xd(x) = \mfrac{x}{2} \left( \log 2 \pi + \gamma - 1 \right)  + \mfrac{1}{2} \left( G_{2,0,2}(x) + \mfrac{1}{x} \, G_{2,2,2}(x) \right)  + O^{\star}  \left( \mfrac{1}{2} \right).$$
\end{lem}

\begin{proof}
The Dirichlet hyperbola principe gives
\begin{align}
   S_{\sigma}(x) &= \sum_{n \leqslant \sqrt{x}} \ \sum_{m \leqslant x/n} (m+n) \left( 1 - \mfrac{x}{mn} \right) - \sum_{n \leqslant \sqrt{x}} n \left( \left \lfloor \sqrt{x} \right \rfloor - \mfrac{x}{n} H \left( \sqrt{x}\right)  \right) \notag \\
   &= \sum_{n \leqslant \sqrt{x}} \ \sum_{m \leqslant x/n} (m+n) \left( 1 - \mfrac{x}{mn} \right) - \tfrac{1}{2} \left \lfloor \sqrt{x} \right \rfloor^{2} \left \lfloor \sqrt{x} + 1 \right \rfloor + x \left \lfloor \sqrt{x} \right \rfloor H \left( \sqrt{x}\right). \label{eq:step_1}
\end{align}
Reporting \eqref{eq:sum_1} and \eqref{eq:sum_2} into \eqref{eq:step_1} yields
\begin{align}
   S_{\sigma}(x) &= - \mfrac{x^{2}}{2} \sum_{n \leqslant \sqrt{x}} \mfrac{1}{n^2} - \tfrac{1}{12} \left \lfloor \sqrt{x} \right \rfloor  \left( 6 \left \lfloor \sqrt{x} \right \rfloor^{2}+9 \left \lfloor \sqrt{x} \right \rfloor +4 \right)  - x  \left \lfloor \sqrt{x} \right \rfloor \Bigl( \log x + \gamma - 1 - H \left( \sqrt{x}\right) \Bigr) \notag \\
   & \hspace*{1cm} + \mfrac{x}{2} H \left( \sqrt{x} \right) + x \sum_{n \leqslant \sqrt{x}} \log n + \mfrac{1}{2} \, G_{2,0,2}(x) + \mfrac{1}{2x} \, G_{2,2,2}(x) + O^{\star} \left( \mfrac{1}{\pi^3} \right). \label{eq:step_2}
\end{align}
Furthermore, by \eqref{eq:harm_1}, we have
\begin{equation}
   \mfrac{x}{2} H(\sqrt{x}) = \mfrac{x}{2} \left( \log \sqrt{x} + \gamma \right) - \mfrac{\sqrt{x} \psi(\sqrt{x})}{2} + O^{\star} \left( \mfrac{1}{16}\right) \label{eq:harm_3} 
\end{equation}
so that reporting \eqref{eq:harm_3} and \eqref{eq:Stirling+harm+1/n^2} into \eqref{eq:step_2}, we get
\begin{align*}
   S_{\sigma}(x) &= - \mfrac{x^{2}}{2} \zeta(2) + \mfrac{x^{3/2}}{2} + \mfrac{x}{2} \left( \log x + \log 2 \pi + \gamma - 1 \right) + \mfrac{\sqrt{x}}{12} - \mfrac{1}{12} \left \lfloor \sqrt{x} \right \rfloor  \left( 6 \left \lfloor \sqrt{x} \right \rfloor^{2}+9 \left \lfloor \sqrt{x} \right \rfloor +4 \right) - \mfrac{3x}{2} \psi(\sqrt{x}) \\
   & \hspace*{1cm} + \mfrac{3}{2}\sqrt{x} B_2 \left( \left\lbrace \sqrt{x} \right\rbrace \right) + \mfrac{1}{2} \, G_{2,0,2}(x) + \mfrac{1}{2x} \, G_{2,2,2}(x) +  O^{\star}  \left( \mfrac{5}{\pi^{3}} + \mfrac{1}{12} + \mfrac{1}{16} \right)
\end{align*}
and the contribution of the term $ \frac{1}{12} \left \lfloor \sqrt{x} \right \rfloor  \left( 6 \left \lfloor \sqrt{x} \right \rfloor^{2}+9 \left \lfloor \sqrt{x} \right \rfloor +4 \right) $, along with $\psi(\sqrt{x})^2 = B_2 \left( \{ \sqrt{x} \} \right) + \frac{1}{12}$ again, provides
\begin{align*}
   S_{\sigma}(x) &= - \mfrac{x^{2}}{2} \zeta(2) + \mfrac{x}{2} \left( \log x + \log 2 \pi + \gamma - 1 \right)  + \underbrace{\tfrac{1}{2} \psi(\sqrt{x})B_2 \left( \left\lbrace \sqrt{x} \right\rbrace  \right)}_{| \dotsb | \leqslant \frac{1}{12}} + \tfrac{1}{24} \\
   & \hspace*{1cm} + \mfrac{1}{2} \, G_{2,0,2}(x) + \mfrac{1}{2x} \, G_{2,2,2}(x) + O^{\star}  \left( \mfrac{5}{\pi^{3}} + \mfrac{1}{12} + \mfrac{1}{16} \right) \\
   &= - \mfrac{x^{2}}{2} \zeta(2) + \mfrac{x}{2} \left( \log x + \log 2 \pi + \gamma - 1 \right)  \\
   & \hspace*{1cm} + \mfrac{1}{2} \left( G_{2,0,2}(x) + \mfrac{1}{x} \, G_{2,2,2}(x) \right) + O^{\star}  \left( \mfrac{5}{\pi^{3}} + \mfrac{13}{48}  \right) \\
   &= M_{\sigma}(x) - x m_{\sigma}(x) + \mfrac{x}{2} \left( \log 2 \pi + \gamma - 1 \right) + \mfrac{1}{2} \left( G_{2,0,2}(x) + \mfrac{1}{x} \, G_{2,2,2}(x) \right) + O^{\star}  \left( \mfrac{5}{\pi^{3}} + \mfrac{13}{48} \right). \\
\end{align*}
We complete the proof with $ \frac{5}{\pi^{3}} + \frac{13}{48} < \frac{1}{2} $. \end{proof}

\subsection{Proof of Theorem~\ref{th:sigma}}

\begin{proof}
By \eqref{eq:borne_G}, we have for all $x \geqslant 1$
$$\left | \mfrac{1}{2} \left(G_{2,0,2}(x) + \mfrac{1}{x} \, G_{2,2,2}(x) \right)  \right | < \mfrac{1}{2} \Biggl( \left( \np{17.94} \, x^{3/10} + \np{4.04} \, x^{1/4} \right) \left( \mfrac{\log x}{\log 4} + 1 \right) + \np{47.76} \, x^{1/4} \left( \mfrac{\log x}{\log 256} + 1 \right) \Biggr) $$
and since $x \geqslant 650$, we derive
$$\left | \mfrac{1}{2} \left(G_{2,0,2}(x) + \mfrac{1}{x} \, G_{2,2,2}(x) \right) \right | < \left( 8 \,  x^{3/10} + 10 \, x^{1/4} \right) \log x < \np{15.25} \, x^{3/10} \log x.$$
Reporting this estimate into Lemma~\ref{le:main} yields the desired result.
\end{proof}

\subsection{Proof of Corollary~\ref{cor:sigma} and~\ref{cor:sigma_2}}

\begin{proof}
By partial summation we derive
$$ \sdfrac{\D(x)}{x}-d(x)= - \mfrac{1}{2} \log x + \mfrac{\zeta(2)}{2} - \int_1^{x} \sdfrac{\D(t)}{t^2} \diff{t}$$
and using Theorem~\ref{th:sigma} yields, for all $x \geqslant 650$
$$\int_1^{x} \sdfrac{\D(t)}{t^2} \diff{t} = - \mfrac{1}{2} \log x + \mfrac{1}{2} \left(1 + \zeta(2) - \gamma - \log 2 \pi \right) + O^{\star} \left( 16 \, x^{-7/10} \log x\right) $$
and the right-hand side is easily seen to be negative when $x \soe 650$. Similarly, for all $x \geqslant 650$
$$\int_1^{x} \sdfrac{\D(t) + t/2}{t^2} \diff{t} = \int_1^{x} \sdfrac{\D(t)}{t^2} \diff{t} + \mfrac{1}{2} \log x = \mfrac{1}{2} \left(1 + \zeta(2) - \gamma - \log 2 \pi \right) + O^{\star} \left( 16 \, x^{-7/10} \log x\right)$$
implying the convergence of the integral and the result follows by making $x \to \infty$.
\end{proof}

\subsection*{Acknowledgments}

\appendix

\renewcommand{\thesection}{\Alph{section}} 
\makeatletter
\renewcommand\@seccntformat[1]{\appendixname\ \csname the#1\endcsname.\hspace{0.5em}}
\makeatother

\section{Useful explicit estimates}

\begin{thm}[Ostrowski's 2nd mean value theorem]
\label{th:moy_1}
Let $F : \left[ a,b \right] \to \C$ be an intégrable function and $G : \left[ a,b \right] \to \R$ be a monotone function. Then
$$\left| \int_a^b F(t) \, G(t) \diff{t} \right| \leqslant \left| G(a) \right| \, \max_{a \leqslant x \leqslant b} \left| \int_a^x F(t) \diff{t} \right| + \left| G(b) \right| \, \max_{a \leqslant x \leqslant b} \left| \int_x^b F(t) \diff{t} \right|.$$
Furthermore, if $G(a) G(b) \geqslant 0$ and $\left| G(a) \right| \geqslant \left| G(b) \right|$, then the $2$nd term may be omitted.
\end{thm}

\begin{proof}
See \cite[page 141]{ost54}, \cite[p. 301--302]{mitrL} or \cite[p. 136]{bul98}.
\end{proof}

\begin{cor}
\label{cor:moy_2_psi}
Let $f : \left[ a,b \right] \to \R$ be a monotone function on $\left[ a,b \right]$. Then
$$\left| \int_a^b \psi(t) \, f(t) \diff{t} \right| \leqslant \mfrac{1}{8} \Bigl( \left| f(a) \right| + \left| f(b) \right| \Bigr).$$
Furthermore, if $f(a) f(b) \geqslant 0$ and $\left| f(a) \right| \geqslant \left| f(b) \right|$, then the $2$nd term may be omitted.
\end{cor}

\begin{proof}
We use Theorem~\ref{th:moy_1} with $F = \psi$ and $G=f$, yielding
\begin{align*}
   \left| \int_a^b \psi(t) \, f(t)\diff{t} \right| & \leqslant \left| f(a) \right| \, \max_{a \leqslant x \leqslant b} \left| \int_a^x \psi(t) \, \mathrm{d}t \right| + \left| f(b) \right| \, \max_{a \leqslant x \leqslant b} \left| \int_x^b \psi(t) \, \mathrm{d}t \right| \\
   &=  \mfrac{1}{2} \biggl( \left| f(a) \right| \, \max_{a \leqslant x \leqslant b} \left| \psi(x)^2 - \psi(a)^2 \right| + \left| f(b) \right| \, \max_{a \leqslant x \leqslant b} \left| \psi(b)^2 - \psi(x)^2 \right| \biggr) \\
   & \leqslant \mfrac{1}{8} \Bigl( \left| f(a) \right| + \left| f(b) \right| \Bigr)
\end{align*}
where we used $\psi(t)^2 \leqslant \frac{1}{4}$. 
\end{proof}

\section{Explicit Euler-Mclaurin summation formul\ae}

\subsection*{Classical results}

\noindent
For the proofs, see \cite{mac}.

\begin{thm}
Let $k \in \Z_{\geqslant 1}$ and $f \in C^k \left[ 1, x \right]$. Then, for all $x \geqslant 1$, we have
\begin{equation}
   \sum_{n \leqslant x} f(n) = \int_1^x f(t) \, \mathrm{d}t + \sum_{j=1}^k \mfrac{1}{j!}\left\{ (-1)^j B_j\left( \left\{ x\right\} \right) f^{(j-1)}(x) - B_j f^{(j-1)}(1) \right\} + \mfrac{(-1)^{k+1}}{k!} \int_1^x B_k \left( \left\{ t\right\} \right) f^{(k)}(t) \diff{t}. \label{eq:em_fini_0}
\end{equation}
Furthermore, if $f \in C^k \left[ 1,+ \infty \right[$ and if $\displaystyle \int_1^\infty \bigl | f^{(k)}(t) \bigr | \diff{t}$ converges, then there exists a constant $\gamma_{f,k}$ such that
\begin{equation}
   \sum_{n \leqslant x} f(n) = \int_1^x f(t) \, \mathrm{d}t + \gamma_{f,k} + \sum_{j=1}^k \mfrac{(-1)^j}{j!} B_j\left( \left\{ x\right\} \right) f^{(j-1)}(x)  + \mfrac{(-1)^{k}}{k!} \int_x^\infty B_k \left( \left\{ t\right\} \right) f^{(k)}(t) \, \diff{t} \label{eq:em_fini}
\end{equation}
where $\gamma_{f,k}$ is given by $\gamma_{f,k} :=$\begin{small}$\displaystyle - \medop \sum_{j=1}^k \mfrac{B_{j}}{j!} f^{(j-1)}(1) + \frac{(-1)^{k+1}}{k!} \medint \int_1^\infty B_{k} \left( \left\{ t\right\} \right)  f^{(k)}(t) \diff{t}$.\end{small}
Also
\begin{equation}
   \sum_{n > x} f(n) = \int_x^\infty f(t) \, \mathrm{d}t + \sum_{j=1}^k \mfrac{(-1)^{j+1}}{j!} B_j\left( \left\{ x\right\} \right) f^{(j-1)}(x)  + \mfrac{(-1)^{k+1}}{k!} \int_x^\infty f^{(k)}(t) B_k \left( \left\{ t\right\} \right) \diff{t}. \label{eq:em_semi_infini}
\end{equation}
\end{thm}

\subsection*{Euler's and Sonine's versions}

\begin{thm}
\label{th:eulermac}
Let $a < b \in \R$ or $x > 1$ be real numbers.
\begin{enumerate}
   \item[\tri] {\rm (Euler)}. Let $f \in C^1 \left[ a,b\right]$. Then
   $$\sum_{a <n \leqslant b} f(n) = \int_a^b f(t) \, \mathrm{d}t + \psi(a) f(a) - \psi(b) f(b) + \int_a^b \psi \left( t\right) f^{\, \prime}(t) \diff{t}.$$
   \item[\tri] {\rm (Sonine)}. Let $f \in C^1 \left[ 1, + \infty \right[$ such that $f^{\, \prime}$ is ultimately monotone from a certain rank $A_f \geqslant 1$ and $\lim\limits_{t \to + \infty} f^{\, \prime} (t) = 0$. Then, for all $x \geqslant 1$,
   \begin{equation}
      \sum_{n\leqslant x} f(n) =\int_1^x f(t) \, \mathrm{d}t + \gamma_{f,1} - \psi(x) f(x) - \int_x^\infty  \psi(t) f^{\, \prime}(t) \, \diff{t} \label{eq:em_Sonine}
   \end{equation}
where $\gamma_{f,1}$ is given by $\gamma_{f,1} := $\begin{small}$\displaystyle \mfrac{1}{2} f(1) + \medint \int_1^\infty  \psi(t) f^{\, \prime}(t)\, \diff{t}$\end{small}. In particular, for all $x \geqslant A_f$, we have
   \begin{equation}
      \sum_{n\leqslant x} f(n) =\int_1^x f(t) \, \diff{t} + \gamma_{f,1} - \psi(x) f(x) + O^{\star} \left( \mfrac{1}{8} |f^{\, \prime}(x)| \right). \label{eq:em_Sonine_reste}
   \end{equation}
\end{enumerate}
\end{thm}

\begin{proof}
For \eqref{eq:em_Sonine}, use \eqref{eq:em_fini} with $k=1$. For \eqref{eq:em_Sonine_reste}, use Corollary~\ref{cor:moy_2_psi} and the hypotheses $f^{\, \prime}$ monotone on $[A_f, + \infty [$ and $\lim\limits_{t \to + \infty} f^{\, \prime} (t) = 0$. 
\end{proof}

\subsection*{Effective version with explicit error term}

\begin{thm}
\label{th:em_synthèse}
Let $k \in \Z_{\geqslant 1}$ and $f \in C^{k} \left[ 1, + \infty \right[$. Assume that
\begin{enumerate}
   \item[\tri] $\lim\limits_{t \to + \infty} f^{(k-1)} (t) = 0$ ;
   \item[\tri] There exists $A_f \geqslant 1$ such that  $f^{(k)}(t)$ is of constant sign for all $t \geqslant A_f$ ;
   \item[\tri] The integral $\displaystyle \medint \int_1^\infty \bigl | f^{(k)}(t) \bigr |\diff{t}$ converges.
\end{enumerate}
Then, for all $x \geqslant 1$
$$\sum_{n \leqslant x} f(n)  = \int_1^x f(t) \diff{t} + \gamma_{f,k} + \sum_{j=1}^{k-1} \mfrac{(-1)^j}{j!} B_j\left( \left\{ x\right\} \right) f^{(j-1)}(x) + O^{\star} \left( \rho_k \int_x^\infty \bigl| f^{(k)}(t) \bigr| \diff{t} \right)$$
where $\gamma_{f,k}$ is given by $\gamma_{f,k} := $\begin{small}$\displaystyle - \medop \sum_{j=1}^k \mfrac{B_{j}}{j!} f^{(j-1)}(1) + \mfrac{(-1)^{k+1}}{k!} \int_1^\infty  B_{k} \left( \left\{ t\right\} \right) f^{(k)}(t) \diff{t}$\end{small}, and
\begin{small}
\begin{center}
\begin{tabular}{ccccc}
$k$ & $1$ & $2$ & $k \geqslant 3$ odd & $k \geqslant 4$ even \\
& & & & \\
$\rho_k$ & $1$ & $\mfrac{1}{8}$ & $ \mfrac{4}{(2 \pi)^k} $ & $ \mfrac{4 \zeta(k)}{(2 \pi)^k} $
\end{tabular}
\end{center}
Furthermore, if $\displaystyle \sum_{n \geqslant 1} f(n)$ converges, then
$$\sum_{n > x} f(n) = \int_x^{\infty} f(t) \, \mathrm{d}t +  \sum_{j=1}^{k-1} \frac{(-1)^{j+1}}{j!} B_j\left( \left\{ x\right\} \right) f^{(j-1)}(x) + O^{\star} \left( \rho_k \int_x^\infty \bigl| f^{(k)}(t) \bigr| \, \mathrm{d}t \right).$$
\end{small}
\end{thm}

\begin{proof}
Suppose $t \geqslant A_f \Longrightarrow f^{(k)}(t) \geqslant 0$ and let $\varepsilon > 0$. Since $\lim\limits_{t \to + \infty} f^{(k-1)} (t) = 0$, there exists $B_\varepsilon > 0$ such that, for all $t \geqslant B_\varepsilon$, we have $\bigl| f^{(k-1)} (t) \bigr | \leqslant \varepsilon$. Then let $z > y \geqslant \max \left(A_f, B_\varepsilon \right)$. Since $f^{(k)} \geqslant 0$ and is continuous, we have
$$\left| \int_y^z B_k (\{t\}) f^{(k)}(t) \diff{t} \right| \leqslant M_k \int_y^z f^{(k)}(t) \diff{t} = M_k \left( f^{(k-1)}(z)-f^{(k-1)}(y) \right)  \leqslant 2 M_k \varepsilon$$
where $M_k := \underset{t \in \R}{\max} \, \bigl| B_k (\{t\}) \bigr|$, so that the integral $\displaystyle \medint \int_1^{\infty} B_k (\{t\}) f^{(k)}(t) \, \textrm{d}t$ converges by Cauchy's criteria. We then apply \eqref{eq:em_fini_0}, yielding
\begin{small}
\begin{align*}
   \sum_{n\leqslant x} f(n) &= \int_1^x f(t) \diff{t} + \sum_{j=1}^k \mfrac{1}{j!}\left\{ (-1)^j B_j\left( \left\{ x\right\} \right) f^{(j-1)}(x) - B_j f^{(j-1)}(1) \right\} + \mfrac{(-1)^{k+1}}{k!} \int_1^x B_{k} \left( \left\{ t\right\} \right) f^{(k)}(t) \diff{t} \\
   &= \int_1^x f(t) \diff{t} - \sum_{j=1}^k \mfrac{B_j}{j!} f^{(j-1)}(1) + \mfrac{(-1)^{k+1}}{k!} \int_1^\infty B_{k} \left( \left\{ t\right\} \right) f^{(k)}(t) \diff{t} \\
   & \hspace*{1cm} + \sum_{j=1}^k \mfrac{(-1)^j}{j!}  B_j\left( \left\{ x\right\} \right) f^{(j-1)}(x) + \mfrac{(-1)^{k}}{k!} \int_x^{\infty} B_{k} \left( \left\{ t\right\} \right) f^{(k)}(t) \diff{t} \\
   & = \int_1^x f(t) \diff{t} + \gamma_{f,k} + \sum_{j=1}^k \mfrac{(-1)^j}{j!} B_j\left( \left\{ x\right\} \right) f^{(j-1)}(x) + \mfrac{(-1)^{k}}{k!} \int_x^{\infty} B_{k} \left( \left\{ t\right\} \right) f^{(k)}(t) \diff{t}.
\end{align*}
\end{small}
On the other hand, since $\lim\limits_{t \to + \infty} f^{(k-1)} (t) = 0$ and since $f^{(k)}$ is continuous, one may write
$$\int_x^\infty B_k \left( \left\{ x\right\} \right) f^{(k)}(t) \diff{t} = B_k \left( \left\{ x\right\} \right) \int_x^\infty f^{(k)}(t) \diff{t} = - B_k \left( \left\{ x\right\} \right) f^{(k-1)}(x)$$
so that
$$\sum_{n\leqslant x} f(n) = \int_1^x f(t) \diff{t} + \gamma_{f,k} + \sum_{j=1}^{k-1} \mfrac{(-1)^j}{j!}  B_j\left( \left\{ x\right\} \right) f^{(j-1)}(x) + \mfrac{(-1)^{k}}{k!} \int_x^\infty \Bigl( B_k \left( \left\{ t\right\} \right) - B_k \left( \left\{ x\right\} \right) \Bigr)f^{(k)}(t) \diff{t}.$$
We complete the proof with:
\begin{enumerate}
   \item[\tri] $\left| B_1 \left( \left\{ t\right\} \right) - B_1 \left( \left\{ x\right\} \right) \right| = \left| \{t\} - \{x\} \right| \leqslant 1$ ;
   \item[\tri] $\left| B_2 \left( \left\{ t\right\} \right) - B_2 \left( \left\{ x\right\} \right) \right| = \bigl| \{t\}(\{t\}-1) - \{x\}(\{x\}-1) \bigr| \leqslant \frac{1}{4}$, since $0 \leqslant \{t\}(\{t\}-1) \leqslant \frac{1}{4}$ ;
   \item[\tri] If $k \geqslant 3$, $ \bigl| B_k \left( \left\{ t\right\} \right) - B_k \left( \left\{ x\right\} \right) \bigr| \leqslant \mfrac{4 \eta(k) k!}{(2 \pi)^k}$ by \cite[Lemma~9]{bor26}, $\eta(k)$ being given in Proposition~\ref{pro:C_W}.
\end{enumerate}
For the sum $\displaystyle \medop \sum_{n > x} f(n) $, the computations are similar using \eqref{eq:em_semi_infini} instead of \eqref{eq:em_fini_0}.
\end{proof}

\end{document}